\documentclass[12pt, reqno]{amsart}
\usepackage{amsmath, amsthm, amscd, amsfonts, amssymb, graphicx, color}
\usepackage[bookmarksnumbered, colorlinks, plainpages]{hyperref}
\hypersetup{colorlinks=true,linkcolor=red, anchorcolor=green, citecolor=cyan, urlcolor=red, filecolor=magenta, pdftoolbar=true}
\input{mathrsfs.sty}

\newtheorem{theorem}{Theorem}[section]
\newtheorem{lemma}[theorem]{Lemma}

\newtheorem{cor}[theorem]{Corollary}
\theoremstyle{definition}

\theoremstyle{remark}
\newtheorem{remark}[theorem]{\bf{Remark}}
\numberwithin{equation}{section}
\usepackage{physics}
\begin{document}

\title [Estimations of Euclidean operator radius] 
       {Estimations of Euclidean operator radius}

\author[P. Bhunia, S. Jana, K. Paul] {Pintu Bhunia, Suvendu Jana, Kallol Paul}

\address{(Bhunia) Department of Mathematics, Indian Institute of Science, Bengaluru-560012, Karnataka, India} \email{pintubhunia5206@gmail.com}

\thanks{Dr. Pintu Bhunia would like to thank SERB, Govt. of India for the financial support in the form of National Post Doctoral Fellowship (N-PDF, File No. PDF/2022/000325) under the mentorship of Professor Apoorva Khare}

\address{(Jana) Department of Mathematics, Mahishadal Girls' College, Purba Medinipur 721628, West Bengal, India} \email{janasuva8@gmail.com} 

\address{(Paul) Department of Mathematics, Jadavpur University, Kolkata 700032, West Bengal, India} \email{kalloldada@gmail.com}

\renewcommand{\subjclassname}{\textup{2020} Mathematics Subject Classification}

\maketitle

\begin{abstract}
We develop several Euclidean operator radius bounds for the product of two $d$-tuple operators using  positivity criteria of a $2\times 2$ block matrix whose entries are $d$-tuple operators. From these bounds, by using the polar decomposition of operators, we obtain  Euclidean operator radius bounds for $d$-tuple operators. Among many other interesting bounds, it is shown that 
\begin{eqnarray*}
	w_e(\mathbf{A})	&\leq&\frac1{\sqrt2} \|\mathbf{A}\|^{1/2}\sqrt{\left\|\sum_{k=1}^{d} (|A_k|+|A_k^*|)\right\|},
\end{eqnarray*} 
	where $w_e(\mathbf{A})$ and $\|\mathbf{A}\|$ are the Euclidean operator radius and the Euclidean operator norm, respectively, of a $d$-tuple operator $\mathbf{A}=(A_1,A_2, \ldots,A_d).$
	Further, we develop an upper bound for the Euclidean operator radius of $n\times n$ operator matrix whose entries are $d$-tuple operators. In particular, it is proved that if $\begin{bmatrix}
		\mathbf{A_{ij}}
	\end{bmatrix}_{n\times n}$ is an $n\times n$ operator matrix  then
	$$ w_e\left(
	\begin{bmatrix}
		\mathbf{A_{ij}}
	\end{bmatrix}_{n\times n}\right)\leq w \left(\begin{bmatrix}
		a_{ij}
	\end{bmatrix}_{n\times n}\right),$$  where each $\mathbf{A_{ij}}$ is a $d$-tuple operator, $1\leq i,j\leq n$, 
$a_{ij}=w_e(\mathbf{A_{ij}})\, \textit{ if i=j}$,  $a_{ij}=
	\sqrt{w_e\left(|\mathbf{A_{ji}|}+|\mathbf{A_{ij}^*}|\right)w_e\left(|\mathbf{A_{ij}|}+|\mathbf{A_{ji}^*}|\right)}\,\textit{ if $i<j$}$, and $a_{ij}=
	0\,\textit{ if $i>j$}.$
Other related applications are also discussed.

\end{abstract}

\section{Introduction}
\noindent 
 Let $ \mathbb{B}(\mathscr{H})$ denote the $C^*$-algebra of all bounded linear operators on a complex Hilbert space $\mathscr{H}$ with inner product $\langle \cdot,\cdot \rangle $ and  $\|\cdot\|$ be the corresponding norm. Let $A\in\mathbb{B}(\mathscr{H})$ and let  
 $|A|=(A^*A)^{1/2}$, where $A^*$ is the adjoint of $A.$
 The numerical radius  of $A$ is defined as $$w(A)=\underset{\|x\|=1}\sup|\langle Ax,x\rangle|.$$ 
 It is well known that $w(\cdot): \mathbb{B}(\mathscr{H}) \to \mathbb{R}$ defines a norm and satisfies the following relation
  $$\frac12 \|A\|\leq w(A) \leq \|A\|.$$ 
 Due to importance of the numerical radius in understanding various analytic and geometric properties of bounded linear operators, the above bounds  have been studied and improved by mathematicians over the years, interested reader can see  \cite{Bhunia_LAMA_2022,Bhunia_LAA_2021,Bhunia_AM_2021,Bhunia_RIM_2021,Bhunia_BSM_2021,Bhunia3,Bhunia_LAA_2019,Kittaneh_LAMA_2023, Kittaneh_STD_2005,Kittaneh_2003} and the references therein. Various generalizations of the numerical radius have also been studied, see \cite{BDMP,P, GW}. The Euclidean operator radius of $d$-tuple operators is one such generalization. 
 Let $\mathbf{A}=(A_1,A_2,\ldots,A_d)$ be a $d$-tuple operator in $  \mathbb{B}^d(\mathscr{H})= \mathbb{B}(\mathscr{H}) \times \mathbb{B}(\mathscr{H}) \times \ldots \times \mathbb{B}(\mathscr{H})$ ($d$ times)
and let $\langle \mathbf{A}x,y\rangle=\left( \langle A_1 x,y\rangle,\langle A_2  x,y\rangle,\ldots,\langle A_d x,y\rangle\right)\in\mathbb{C}^d$ for all $x,y\in \mathscr{H}$. The Euclidean operator radius and the Euclidean operator norm of $\mathbf{A}$ are defined respectively as
\begin{eqnarray*}
	w_e (\mathbf{A}) &=&\sup \left\lbrace \left( \sum_{k=1}^{d} |\langle A_k x,x\rangle|^2 \right)^\frac{1}{2}:  x\in\mathscr{H}, \|x\|=1\right\rbrace
\end{eqnarray*}
and
\begin{eqnarray*}
	\|\mathbf{A}\|&=&\sup \left\lbrace \left( \sum_{k=1}^{d} \|A_kx\|^2 \right)^\frac12:x\in\mathscr{H}, \|x\|=1\right\rbrace.
\end{eqnarray*}

A $d$-tuple operator $\mathbf{A}=(A_1,A_2,\ldots,A_d)\in\mathbb{B}^d(\mathscr{H})$ is said to be positive if  each $A_k$ is positive for all $k=1,2,\ldots,d$. 
We write $|\mathbf{A}|^t=(|A_1|^t,|A_2|^t,\ldots,|A_d|^t)$ for $t>0$,  
and $\alpha \mathbf{A}=(\alpha A_1,\alpha A_2, \ldots,\alpha A_d)$ for any scalar $\alpha \in\mathbb{C}$ .
For $\mathbf{B}=(B_1,B_2,\ldots,B_d)\in\mathbb{B}^d(\mathscr{H})$, we write $\mathbf{AB}=(A_1B_1,A_2B_2,\ldots,A_dB_d)$, $\mathbf{A}+\mathbf{B}=(A_1+B_1,A_2+B_2,\ldots,A_d+B_d).$ 
Let $\mathbf{A_{ij}}=(A_{ij}^1,A_{ij}^2,\ldots,A_{ij}^d)\in\mathbb{B}^d(\mathscr{H}),$ $1\leq i,j\leq n.$
Then the $n\times n$ operator matrix, whose  entries are $d$-tuple operators $\mathbf{A_{ij}}$,   is defined as $$\begin{bmatrix}
	\mathbf{A_{ij}}
\end{bmatrix}_{n\times n}= \left( \begin{bmatrix}
	{A_{ij}^1}
\end{bmatrix}_{n\times n}, \begin{bmatrix}
	{A_{ij}^2}
\end{bmatrix}_{n\times n}, \ldots, \begin{bmatrix}
	{A_{ij}^d}
\end{bmatrix}_{n\times n} \right)\in \mathbb{B}^d\left(\oplus_{i=1}^d \mathscr{H} \right).$$

It is well known that the Euclidean operator radius defines a norm on $  \mathbb{B}^d(\mathscr{H})$ and satisfies the following relation
 \begin{eqnarray}
	\frac{1}{2\sqrt{d}} \| \mathbf{A}\| \leq w_e(\mathbf{A})\leq \|\mathbf{A}\|.
	\label{T1}\end{eqnarray}
which can be found in \cite{SPK1,P}.
Note that the constants $\frac{1}{2\sqrt{d}}$ and $1$ are best possible. Some improvements of (\ref{T1}) and related inequalities  have been studied in  \cite{SD,SPK,MSS,Sahoo}.   For $d=1,$ the above inequality (\ref{T1}) reduces to the well-known classical numerical radius inequality.  Recently in \cite{Bhunia1,Bhunia2} author has developed bounds for the numerical radius of bounded linear operators. Motivated by this 
 we here obtain  several upper bounds for the Euclidean operator radius of $d$-tuple operators and operator matrices, which generalize and improve the classical numerical radius bounds. 
\noindent 

This article is organized as follows: In Section 2, by using  positivity of a $2\times 2$ block matrix,  whose entries are $d$-tuple operators, we obtain several upper bounds of the Euclidean operator radius of the product of two $d$-tuple operators. From these estimations and by using the polar decomposition, we develop several upper bounds for the Euclidean operator radius of $d$-tuple operators.  In Section 3, we develop an upper bound for the Euclidean operator radius of $n\times n$ operator matrix whose entries are $d$-tuple operators. As an application of these bounds we derive  upper bounds for the Euclidean operator radius of $d$-tuple operators.


\section{Euclidean operator radius of $d$-tuple operators}

We begin this section with the following lemmas. First lemma is known as McCarthy inequality.

\begin{lemma}\cite{M}
Let $A\in\mathbb{B}(\mathscr{H})$ be positive, and $x\in\mathscr{H}$ with $\|x\|=1$. Then  
  $$ \langle Ax, x \rangle^p \leq \langle A^p x, x \rangle,$$  for all $p\geq 1.$
\label{lm1}\end{lemma}

Second lemma is known as Buzano's inequality, which is an extension of Schwarz's inequality.

\begin{lemma} \cite{B}
Let    $x,y,z\in\mathscr{H}$ be such that $\|z\|=1.$ Then 
    $$ |\langle x, z \rangle \langle z, y \rangle|\leq \frac{\|x\|\|y\|+|\langle x , y \rangle|}{2}.$$

\label{lm2}\end{lemma}

Third lemma is on non-negative real numbers and is known as Bohr’s inequality.

\begin{lemma}\cite{VK}
Let $a_k\geq 0$ for $k=1,2,\ldots,n$. Then 
$$ \left(\sum_{k=1}^{n} a_k\right)^p\leq n^{p-1}\sum_{k=1}^{n} a_k^p,$$
for all $p\geq1.$
\label{lm4}\end{lemma}

The next  lemma  involves $2\times 2$ operator matrix, whose  entries are $d$-tuple operators.

\begin{lemma}
 Let $\mathbf{A}=(A_1,A_2, \ldots,A_d), \mathbf{B}=(B_1,B_2, \ldots,B_d), \mathbf{C}=(C_1,C_2, \ldots,C_d)\in\mathbb{B}^d(\mathscr{H})$,
 where $\mathbf{A}$ and $\mathbf{B}$ are positive
 (i.e., $A_k$ and $ B_k$ are positive for each $k=1,2,\ldots,d$). If  
$ \begin{bmatrix}
 \mathbf{A} & \mathbf{C^*} \\
 \mathbf{C} & \mathbf{B} \\
\end{bmatrix}$ is positive, then 

$$\|\langle \mathbf{C}x,y\rangle\|^2\leq \sum_{k=1}^{d} \langle A_k x, x\rangle \langle B_k y, y\rangle, \, \text{ for all $x,y\in\mathscr{H}$.} $$

\label{lm3}\end{lemma}

\begin{proof}
Take $x,y\in \mathscr{H}.$ We have
   \begin{eqnarray*}
\|\langle \mathbf{C}x,y\rangle\|^2&=& \sum_{k=1}^{d}|\langle C_k x,y\rangle|^2\\
&=& \sum_{k=1}^{d} \left|\left\langle \begin{bmatrix}
A_k & C_k^* \\
C_k & B_k\\
\end{bmatrix}\begin{bmatrix}
x \\
0 \\
\end{bmatrix},\begin{bmatrix}
0 \\
y \\
\end{bmatrix}\right\rangle\right|^2\\
&\leq&  \sum_{k=1}^{d} \left\langle \begin{bmatrix}
A_k & C_k^* \\
C_k & B_k\\
\end{bmatrix}\begin{bmatrix}
x \\
0 \\
\end{bmatrix},\begin{bmatrix}
x \\
0 \\
\end{bmatrix}\right\rangle \left \langle \begin{bmatrix}
A_k & C_k^* \\
C_k & B_k\\
\end{bmatrix}\begin{bmatrix}
0 \\
y \\
\end{bmatrix},\begin{bmatrix}
0 \\
y \\
\end{bmatrix} \right\rangle\\
&&\,\,\,(\textit{by Cauchy-Schwarz inequality for positive operators})\\
&=& \sum_{k=1}^{d} \langle A_k x, x\rangle \langle B_k y, y\rangle.
\end{eqnarray*}

\end{proof}

We are now in a position to prove our first theorem.

\begin{theorem}\label{th1}
Let $\mathbf{A}=(A_1,A_2, \ldots,A_d),\mathbf{B}=(B_1,B_2, \ldots,B_d),  \mathbf{C}=(C_1,C_2, \ldots,C_d) \in\mathbb{B}^d(\mathscr{H})$, where $\mathbf{A}$ and $\mathbf{B}$ are positive (i.e., $A_k$ and $ B_k$ are positive for each $k=1,2,\ldots,d$). If
$ \begin{bmatrix}
 \mathbf{A} & \mathbf{C^*} \\
 \mathbf{C} & \mathbf{B} \\
\end{bmatrix}$ is positive, then\\
(i) $ {w_e(\mathbf{C})} \leq \sqrt{\frac12 \left\|\sum_{k=1}^{d} \left(A_k^2+B_k^2\right)\right\|}.$\\
(ii) $ w_e(\mathbf{C})\leq \sqrt{\frac12 \|\mathbf{A}\| \|\mathbf{B}\|+\frac{\sqrt{d}}{2}w_e(\mathbf{AB})}.$\\
(iii) $ w_e(\mathbf{C})\leq \sqrt{\frac14 \left\|\sum_{k=1}^{d} \left(A_k^2+B_k^2\right)\right\|+\frac{\sqrt{d}}{2}w_e(\mathbf{AB})}.$ 
\end{theorem}
\begin{proof}
Let $x\in \mathscr{H}$  with $\|x\|=1.$\\
(i)  From Lemma \ref{lm3}, we have 
\begin{eqnarray*}
\hspace{1cm}  \|\langle \mathbf{C}x,x\rangle\|^2&=&\sum_{k=1}^{d} \big|\langle C_k x, x\rangle\big|^2\leq \sum_{k=1}^{d} \langle A_k x, x\rangle \langle B_k x, x\rangle\\
&\leq& \sum_{k=1}^{d} \frac12 \left(\langle A_k x, x\rangle^2+ \langle B_k x, x\rangle^2\right)\\
&\leq&\sum_{k=1}^{d} \frac12 \left(\langle A_k^2 x, x\rangle+ \langle B_k^2 x, x\rangle\right)\,\,\,(\textit{using Lemma \ref{lm1}})\\
&=&\sum_{k=1}^{d} \frac12 \langle (A_k^2 + B_k^2) x, x\rangle\\
&=&\frac12 \left\langle sum_{k=1}^{d}(A_k^2 + B_k^2) x, x \right\rangle\\
&\leq&\frac12 \left\|\sum_{k=1}^{d} (A_k^2+B_k^2)\right\|.
\end{eqnarray*}
This holds for all $x\in \mathscr{H}$ with $\|x\|=1$ and so taking supremum we get the desired inequality.

(ii) Again from Lemma \ref{lm3}, we get
\begin{eqnarray*}
\hspace{1cm}  \|\langle \mathbf{C}x,x\rangle\|^2&=&\sum_{k=1}^{d} \big|\langle C_k x, x\rangle\big|^2\leq \sum_{k=1}^{d} \langle A_k x, x\rangle \langle B_k x, x\rangle\\
&\leq&\sum_{k=1}^{d}\frac12\left(\|A_kx\|\|B_kx\|+|\langle A_kx, B_kx\rangle|\right)\,\,\,(\textit{using Lemma \ref{lm2}})\\
&\leq& \frac12 \left(\sum_{k=1}^{d}\|A_kx\|^2\right)^{\frac12}\left(\sum_{k=1}^{d}\|B_kx\|^2\right)^{\frac12}+\frac12\sum_{k=1}^{d}| \langle A_kx, B_kx\rangle|\\
&&\,\,\,\,\,\,\,\,(\textit{using Cauchy-Schwarz inequality})\\
&\leq&\frac12 \|\mathbf{A}\|\|\mathbf{B}\|+\frac{\sqrt{d}}{2}\left(\sum_{k=1}^{d}| \langle x, A_kB_kx\rangle|^2\right)^{\frac12}\,\,\,\,\,\,\,\,(\textit{using Lemma \ref{lm4}})\\
&\leq& \frac12 \|\mathbf{A}\|\|\mathbf{B}\|+\frac{\sqrt{d}}{2}w_e(\mathbf{AB}).
\end{eqnarray*}
Taking supremum over all $x\in \mathscr{H}$, $\|x\|=1$, we get the desired result.\\
(iii) From Lemma \ref{lm3}, we have 
\begin{eqnarray*}
  \|\langle \mathbf{C}x,x\rangle\|^2&=&\sum_{k=1}^{d} \big|\langle C_k x, x\rangle\big|^2\leq \sum_{k=1}^{d} \langle A_k x, x\rangle \langle B_k x, x\rangle\\
&\leq&\sum_{k=1}^{d}\frac12\left(\|A_kx\|\|B_kx\|+|\langle A_kx, B_kx\rangle|\right)\,\,\,(\textit{using Lemma \ref{lm2}})\\
&\leq& \frac14\sum_{k=1}^{d}(\|A_kx\|^2+\|B_kx\|^2)+\frac12\sum_{i=1}^{d}|\langle A_kx, B_kx\rangle|\\
&=& \frac14\sum_{k=1}^{d}\left(\langle  A_k^2x,x\rangle+\langle B_k^2x,x\rangle\right)+\frac12\sum_{i=1}^{d}|\langle A_kx, B_kx\rangle|\\
&=&\frac14\sum_{k=1}^{d}\langle (A_k^2+B_k^2)x,x\rangle+\frac12\sum_{k=1}^{d}|\langle A_kx, B_kx\rangle|\\
&\leq& \frac14 \left\langle \sum_{k=1}^{d}\left(A_k^2+ B_k^2\right)x,x\right \rangle+\frac{\sqrt{d}}{2}\left(\sum_{k=1}^{d}|\langle x,A_kB_kx \rangle|^2\right)^{\frac12}
(\textit{using Lemma \ref{lm4}})\\
&\leq&\frac14 \left\|\sum_{k=1}^{d} \left(A_k^2+B_k^2\right)\right\|+\frac{\sqrt{d}}{2}w_e(\mathbf{AB})
\end{eqnarray*}
Taking supremum over all $x\in \mathscr{H}$, $\|x\|=1$ we obtain the desired bound.
\end{proof}

As an application of  Theorem \ref{th1}, we obtain the following bounds  for  the Euclidean operator radius of the  product of two $d$-tuple operators.

\begin{cor}
Let $\mathbf{B}=(B_1,B_2, \ldots,B_d), \mathbf{C}=(C_1,C_2, \ldots,C_d)\in\mathbb{B}^d(\mathscr{H}),$  then \\
(i) $ w_e(\mathbf{BC})\leq \sqrt{\frac12 \left\|\sum_{k=1}^{d} \left(|B_k^*|^4+|C_k|^4\right)\right\|}.$\\
(ii) $ w_e(\mathbf{BC})\leq \sqrt{\frac12\|\mathbf{BB^*}\|\|\mathbf{C^*C}\|+\frac{\sqrt{d}}{2}w_e(\mathbf{B(CB)^*C})}$.\\
(iii) $ w_e(\mathbf{BC})\leq \sqrt{\frac14\|\sum_{k=1}^{d} |B_k^*|^4+|C_k|^4\|+\frac{\sqrt{d}}{2}w_e(\mathbf{B(CB)^*C})}$.
\label{cor1.1}
\end{cor} 

\begin{proof}
Observe  that $\begin{bmatrix}
\mathbf{BB^*} & \mathbf{BC}\\
\mathbf{C^*B^*} & \mathbf{C^*C}
\end{bmatrix}$ is positive. Using this positive operator matrix in Theorem \ref{th1}, we obtain the desired inequalities.
\end{proof}

In the next results, we obtain an upper bound for the Euclidean operator radius of  $d$-tuple operators.

\begin{cor} Let $\mathbf{A}=(A_1,A_2, \ldots, A_d) \in {\mathbb{B}^d(\mathscr{H})}$, then 
$$ w_e(\mathbf{A})\leq \sqrt{\frac12 \left\|\sum_{k=1}^{d}\left( |A_k^*|^{4(1-t)}+|A_k|^{4t}\right)\right\|},$$
for all $t$, $0\leq t \leq 1.$
 In particular,
$$  w_e(\mathbf{A})\leq \sqrt{\frac12  \left\|\sum_{k=1}^{d} \left(|A_k^*|^2+|A_k|^2\right)\right\|}.$$

\label{th2}\end{cor}
\begin{proof}
Let $A_k=U_k|A_k|$ be the polar decomposition of $A_k$, for each $k=1,2,\ldots,d.$ Considering $B_k=U_k|A_k|^{1-t}$ and $C_k=|A_k|^t$ in  Corollary \ref{cor1.1} (i), we get 
$$ w_e(\mathbf{T})\leq \sqrt{\frac12 \left\|\sum_{k=1}^{d}\left( (U_k|A_k|^{2(1-t)}U_k^*)^2+|A_k|^{4t}\right)\right\|}.$$
Since $ U_k|A_k|^{2(1-t)}U_k^*=|A_k^*|^{2(1-t)},$ we obtain the first inequality. And the second inequality follows by considering $t=\frac12.$
\end{proof}



\begin{cor}\label{th3}
Let $\mathbf{A}=(A_1,A_2, \ldots, A_d) \in {\mathbb{B}^d(\mathscr{H})}$, then 
\begin{eqnarray*}
 w_e(\mathbf{A})&\leq&   \sqrt{\frac12 \left\||\mathbf{A^*}|^{2(1-t)}\right\|\left\||\mathbf{A}|^{2t}\right\|+\frac{\sqrt{d}}{2}w_e\left(\mathbf{|A|}^{2t}\mathbf{|A^*|}^{2(1-t)}\right)}\\
&=& \sqrt{\frac12\sqrt{\left\|\sum_{k=1}^{d} |A_k^*|^{4(1-t)}\right\|}\sqrt{\left\|\sum_{k=1}^{d}|A_k|^{4t}\right\|}+\frac{\sqrt{d}}{2}w_e\left(\mathbf{|A|}^{2t}\mathbf{|A^*|}^{2(1-t)}\right)}, 
\end{eqnarray*}
for all $t$, $0\leq t \leq 1.$
 In particular,
 \begin{eqnarray*}
 w_e(\mathbf{A})&\leq& \sqrt{\frac12\|\mathbf{A^*}\|\|\mathbf{A}\|+\frac{\sqrt{d}}{2}w_e(\mathbf{|A|}\mathbf{|A^*|})}\\
&=&\sqrt{\frac12\sqrt{\left\|\sum_{k=1}^{d} |A_k^*|^{2}\right\|}\sqrt{\left\|\sum_{k=1}^{d}|A_k|^{2}\right\|}+\frac{\sqrt{d}}{2}w_e(\mathbf{|A|}\mathbf{|A^*|})}.
\end{eqnarray*}
\end{cor}

\begin{proof}
Let $A_k=U_k|A_k|$ be the polar decomposition of $A_k$ for each $k=1,2,\ldots,d.$ Considering $B_k=U_k|A_k|^{1-t}$ and $C_k=|A_k|^t$  in  Corollary \ref{cor1.1} (ii) and using similar arguments as Corollary \ref{th2}, we obtain the desired bounds.
\end{proof}

\begin{cor} \label{th4}
	Let $\mathbf{A}=(A_1,A_2, \ldots, A_d) \in {\mathbb{B}^d(\mathscr{H})}$, then 
	$$ w_e(\mathbf{A})\leq \sqrt{\frac14 \left\|\sum_{k=1}^{d}\left( |A_k^*|^{4(1-t)}+|A_k|^{4t}\right)\right\|+\frac{\sqrt{d}}{2}w_e\left(\mathbf{|A|}^{2t}\mathbf{|A^*|}^{2(1-t)}\right)},$$
	for all $t,$ $0\leq t \leq 1.$
	In particular, 
	$$  w_e(\mathbf{A})\leq \sqrt{\frac14  \left\|\sum_{k=1}^{d} \left(|A_k^*|^2+|A_k|^2\right)\right\|+\frac{\sqrt{d}}{2}w_e(\mathbf{|A|}\mathbf{|A^*|})}.$$
\end{cor}
\begin{proof}
	Let $A_k=U_k|A_k|$ be the polar decomposition of $A_k$ for each $k=1,2,\ldots,d.$ Considering $B_k=U_k|A_k|^{1-t}$ and $C_k=|A_k|^t$  in Corollary \ref{cor1.1} (iii) and using the similar argument as  Corollary \ref{th2}, we get the desired results.
\end{proof}

To obtain our next result we need the following lemma.

\begin{lemma}\cite{FK}
Let $A, B \in \mathbb{B}(\mathscr{H})$ be such that $|A|B = B^*|A|$. Let $f$ and $g$ be two non-negative continuous functions on $[0,\infty)$ such that $f (\lambda)g(\lambda) =\lambda$, for all $\lambda \in [0,\infty)$. Then
$$|\langle ABx, y\rangle| \leq r(B)\|f(|A|)x\|\|g(|A^*|)y\|,$$
for all $x, y \in\mathscr{ H}$.
\label{lm5}\end{lemma}

The following result provides  an upper bound for the Euclidean operator radius of the product of two $d$-tuple operators.

\begin{theorem}
Let $\mathbf{B}=(B_1,B_2, \ldots,B_d)$,  $\mathbf{C}=(C_1,C_2, \ldots,C_d)\in \mathbb{B}^d(\mathscr{H})$ be such that
$|\mathbf{B}|\mathbf{C}=\mathbf{C}^*|\mathbf{B}|$ (i.e.,
 $|B_k|C_k=C_k^*|B_k|$ for all $k=1,2,\ldots,d$). If $f,g: [0,\infty)\rightarrow [0,\infty)$ are continuous function with $f(\lambda)g(\lambda)=\lambda$ for all $\lambda\geq0$, then 
 \begin{eqnarray*}
  w_e(\mathbf{B}\mathbf{C})&\leq& {\frac1{\sqrt{2}}\max_{k}\left\{r(C_k)\right\} w_e\left(f^2(|\mathbf{B}|)+i g^2(|\mathbf{B^*}|)\right)}\\
  &\leq&{\frac1{\sqrt2}\max_{k}\left\{r(C_k)\right\} \sqrt{\left\|\sum_{k=1}^{d}\left(f^4(|B_k|)+g^4(|B_k^*|)\right)\right\|}}. \end{eqnarray*}
\label{th9}\end{theorem}
\begin{proof}
Let $x\in \mathscr{H}$ with $\|x\|=1$. Then we get, 
\begin{eqnarray*}
\|\langle \mathbf{BC} x,x\rangle\|^2 &=&\sum_{k=1}^{d} |\langle B_kC_kx,x\rangle|^2\\
&\leq &\sum_{k=1}^{d} r^2(C_k)\|f(|B_k|)x\|^2\|g(|B_k^*|)x\|^2\,\,\,\,\,(\textit{using Lemma \ref{lm5}})\\
&=&\sum_{k=1}^{d} r^2(C_k)\langle f^2(|B_k|)x,x\rangle\langle g^2(|B_k^*|)x,x\rangle\\
&\leq& \frac12\sum_{k=1}^{d} r^2(C_k)\left(\langle f^2(|B_k|)x,x\rangle^2+\langle g^2(|B_k^*|)x,x\rangle^2\right)\\
&=& \frac12\sum_{k=1}^{d} r^2(C_k)|\langle f^2(|B_k|)x,x\rangle+i\langle g^2(|B_k^*|)x,x\rangle|^2\\
&\leq &\frac12\max_{k}\{ r^2(C_k)\} \sum_{k=1}^{d}|\langle (f^2(|B_k|)+i g^2(|B_k^*|))x,x\rangle|^2\\
&\leq&\frac12\max_{k}\{r^2(C_k)\}  w_e^2(f^2(|\mathbf{B}|)+ig^2(|\mathbf{B^*}|)).
\end{eqnarray*}
Therefore, taking supremum over all $x\in \mathscr{H}$, $\|x\|=1$, we obtain the first inequality.
Next, we see that 
\begin{eqnarray*}
&& w_e^2(f^2(|\mathbf{B}|)+ig^2(|\mathbf{B^*}|))\\
&=&\underset{\|x\|=1}\sup\sum_{k=1}^{d}\left(\langle f^2(|B_k|)x,x\rangle^2+\langle g^2(|B_k^*|)x,x\rangle^2\right) \\
&\leq&\underset{\|x\|=1}\sup\sum_{k=1}^{d}\left(\langle f^4(|B_k|)x,x\rangle+\langle g^4(|B_k^*|)x,x\rangle\right)\,\,\,(\textit{using Lemma \ref{lm1}})\\
&=&\underset{\|x\|=1}\sup\sum_{k=1}^{d}\langle (f^4(|B_k|)+ g^4(|B_k^*|))x,x\rangle\\
&=&\underset{\|x\|=1}\sup \left\langle \sum_{k=1}^{d}(f^4(|B_k|)+ g^4(|B_k^*|))x,x\right\rangle\\
&=& \left\|\sum_{k=1}^{d}(f^4(|B_k|)+ g^4(|B_k^*|))\right\|,
\end{eqnarray*}
which gives the second inequality.
\end{proof}

The inequalities in Theorem \ref{th9} include several Euclidean operator radius inequalities for $d$-tuple operators. Some of these are demonstrated in the following corollaries. 

\begin{cor}
Let   $\mathbf{A}=(A_1,A_2, \ldots,A_d)\in \mathbb{B}^d(\mathscr{H})$ and let $f,g: [0,\infty)\rightarrow [0,\infty)$ be continuous functions, where $f(\lambda)g(\lambda)=\lambda$ for all $\lambda\geq 0$. Then  
 \begin{eqnarray*}
 w_e(\mathbf{A})&\leq & \frac1{\sqrt{2}} \max_{k}\left\{\|A_k\|^{t}\right\}w_e\left(f^2(|\mathbf{A}|^{1-t})+ig^2(|\mathbf{A^*}|^{1-t})\right)\\
&\leq&\frac1{\sqrt{2}} \max_{k}\left\{\|A_k\|^{t}\right\} \sqrt{\left\|\sum_{k=1}^{d} f^4(|A_k|^{1-t})+g^4(|A_k^*|^{1-t})\right\|}\\
&\leq& \frac1{\sqrt{2}}\|\mathbf{A}\|^{t} \sqrt{\left\|\sum_{k=1}^{d} f^4(|A_k|^{1-t})+g^4(|A_k^*|^{1-t})\right\|},
  \end{eqnarray*}
 for all $t$, $0\leq t\leq 1.$
\label{th10}\end{cor}
\begin{proof}
Let $A_k=U_k|A_k|$ be the polar decomposition of $A_k$ for each $k=1,2,\ldots,d.$ Considering $B_k=U_k|A_k|^{1-t}$ and $C_k=|A_k|^t$ in Theorem \ref{th9}, we obtain 
 \begin{eqnarray*}
 w_e(\mathbf{A})&\leq&\frac{1}{\sqrt{2}} \max_{k}\{r(|A_k^t|)\}w_e\left(f^2\left(|\mathbf{A}|^{1-t}\right)+ig^2\left(|\mathbf{A^*}|^{1-t}\right)\right)\\
 &= & \frac{1}{\sqrt{2}} \max_{k}\{\|A_k\|^{t}\}w_e\left(f^2\left(|\mathbf{A}|^{1-t}\right)+ig^2\left(|\mathbf{A^*}|^{1-t}\right)\right)\\
&\leq&\frac{1}{\sqrt{2}}\max_{k}\{\|A_k\|^{t}\}\sqrt{\left\|\sum_{k=1}^{d} f^4\left(|A_k|^{1-t}\right)+g^4\left(|A_k^*|^{1-t}\right)\right\|}\\
&\leq&\frac{1}{\sqrt{2}} \|\mathbf{A}\|^{t}\sqrt{\left\|\sum_{i=1}^{d} f^4\left(|
A_k|^{1-t}\right)+g^4\left(|A_k^*|^{1-t}\right)\right\|}.
  \end{eqnarray*}
  Since $\|A_k\|\leq\|\mathbf{A}\|$ for all $ k=1,2,\ldots,d$, the last inequality follows easily. This completes the proof.
\end{proof}

\begin{remark} Suppose $\mathbf{A}=(A_1,A_2, \ldots,A_d)\in \mathbb{B}^d(\mathscr{H}).$\\
(i) Considering $f(\lambda)=\lambda^\alpha$ and $g(\lambda)=\lambda^{1-\alpha}$, $0\leq \alpha \leq 1,$ in Corollary \ref{th10}, we get
\begin{eqnarray*}
 w_e(\mathbf{A})&\leq & \frac1{\sqrt{2}} \max_{k}\{\|A_k\|^{t}\}w_e\left(|\mathbf{A}|^{2\alpha(1-t)}+i|\mathbf{A^*}|^{2(1-\alpha)(1-t)}\right)\\
&\leq&\frac1{\sqrt{2}} \max_{k}\{\|A_k\|^{t}\} \sqrt{ \left\| \sum_{k=1}^{d} \left(|A_k|^{4\alpha(1-t)}+|A_k^*|^{4(1-\alpha)(1-t)}\right)\right\|}\\
&\leq&\frac1{\sqrt2} \|\mathbf{A}\|^{t} \sqrt{\left\|\sum_{k=1}^{d} \left(|A_k|^{4\alpha(1-t)}+|A_k^*|^{4(1-\alpha)(1-t)}\right)\right\|},
  \end{eqnarray*} for all $\alpha,t$,  $0\leq \alpha,t \leq 1.$\\
  (ii) In particular, for $t=\frac12$,
  \begin{eqnarray*}
 w_e(\mathbf{A})&\leq & \frac1{\sqrt{2}} \max_{k}\left\{\|A_k\|^{1/2}\right\}w_e\left(|\mathbf{A}|^{\alpha}+i|\mathbf{A^*}|^{(1-\alpha)}\right)\\
&\leq&\frac1{\sqrt{2}} \max_{k}\{\|A_k\|^{1/2}\}\sqrt{\left\|\sum_{k=1}^{d} \left(|A_k|^{2\alpha}+|A_k^*|^{2(1-\alpha)}\right)\right\|}\\
&\leq&\frac1{\sqrt{2}} \|\mathbf{A}\|^{1/2} \sqrt{\left\|\sum_{k=1}^{d} \left(|A_k|^{2\alpha}+|A_k^*|^{2(1-\alpha)}\right)\right\|},
  \end{eqnarray*} 
for all $\alpha$, $0\leq \alpha \leq 1.$\\
  (iii) In particular, for $\alpha=\frac12$,
  \begin{eqnarray*}
 w_e(\mathbf{A})&\leq & \frac1{\sqrt{2}} \max_{k}\left\{\|A_k\|^{1/2}\right\}w_e(|\mathbf{A}|^{\frac12}+i|\mathbf{A^*}|^{\frac12})\\
&\leq&\frac1{\sqrt{2}} \max_{k} \left\{\|A_k\|^{1/2}\right\} \sqrt{\left\|\sum_{k=1}^{d} (|A_k|+|A_k^*|)\right\|}\\
&\leq&\frac1{\sqrt2} \|\mathbf{A}\|^{1/2}\sqrt{\left\|\sum_{k=1}^{d} (|A_k|+|A_k^*|)\right\|}.
  \end{eqnarray*}  
\end{remark}

Next, we obtain an upper bound for the Euclidean operator radius of product of two $d$-tuple operators.

\begin{theorem}
Let $\mathbf{B}=(B_1,B_2, \ldots,B_d)$,  $\mathbf{C}=(C_1,C_2, \ldots,C_d)\in \mathbb{B}^d(\mathscr{H}).$  Then  $$ w_e(\mathbf{B}\mathbf{C})\leq \frac1{\sqrt2} w_e(|\mathbf{C}|^2+i \mathbf{|B^*|}^2).$$
\label{th7}\end{theorem}
\begin{proof}
Let $x\in \mathscr{H}$ with $\|x\|=1.$ Then we have
\begin{eqnarray*}
\|\langle \mathbf{BC} x,x\rangle\|^2&=& \sum_{k=1}^{d}|\langle B_kC_k x,x\rangle|^2= \sum_{k=1}^{d}|\langle C_k x,B_k^*x\rangle|^2\\
&\leq&  \sum_{k=1}^{d}\| C_k x\|^2\|B_k^*x\|^2= \sum_{k=1}^{d}\langle |C_k|^2 x,x\rangle\langle |B_k^*|^2x,x\rangle\\
&\leq& \frac12  \sum_{k=1}^{d}\left(\langle |C_k|^2 x,x\rangle^2+\langle |B_k^*|^2x,x\rangle^2\right)\\
&=& \frac12  \sum_{k=1}^{d}|\langle |C_k|^2 x,x\rangle+i\langle |B_k^*|^2x,x\rangle|^2\\
&=&\frac12  \sum_{k=1}^{d}|\langle \left( |C_k|^2 +i |B_k^*|^2\right)x,x\rangle|^2\\
&\leq&  \frac12 w_e^2(|\mathbf{C}|^2+i \mathbf{|B^*|}^2).
\end{eqnarray*}
Therefore, taking supremum over $\|x\|=1$, we get desired result.
\end{proof}

Using Theorem \ref{th7}, we obtain the following bounds.

\begin{cor}
 Let $\mathbf{A}=(A_1,A_2, \ldots, A_d) \in \mathbb{B}^d(\mathscr{H})$, then 
\begin{eqnarray*}
w_e(\mathbf{A})&\leq&\frac1{\sqrt{2}} w_e(|\mathbf{A}|^{2t}+i|\mathbf{A^*}|^{2(1-t)})\\
&\leq& \frac1{\sqrt2} \left\|\sum_{k=1}^{d}\left( |A_k^*|^{4(1-t)}+|A_k|^{4t}\right)\right\|^{1/2},
\end{eqnarray*} 
for all $t,\, 0\leq t \leq 1.$
 In particular,
\begin{eqnarray*}
w_e(\mathbf{A})&\leq&\frac1{\sqrt2} w_e(|\mathbf{A}|+i|\mathbf{A}^*|)\\
&\leq& \frac1{\sqrt{2}} \left\|\sum_{k=1}^{d}\left( |A_k^*|^{2}+|A_k|^{2}\right)\right\|^{1/2}.
\end{eqnarray*} 
\label{th8}\end{cor}
\begin{proof}
Let $A_k=U_k|A_k|$ be the polar decomposition of $A_k$ for each $k=1,2,\ldots,d.$ Considering $B_k=U_k|A_k|^{1-t}$ and $C_k=|A_k|^t$  in Theorem \ref{th7}, we get the desired bounds.
\end{proof}

\section{Euclidean operator radius of operator matrices}




In this section, first we develop an upper bound for the Euclidean operator radius of $n\times n$ operator matrix whose entries are $d$-tuple operators.

\begin{theorem}
Let $\mathbb{A}=\begin{bmatrix}
\mathbf{A_{ij}}
\end{bmatrix}_{n\times n}$ be an $n \times n$ operator matrix, where  $\mathbf{A_{ij}}\in\mathbb{B}^d(\mathscr{H}),$ $1\leq i,j\leq n.$ If $f, g : [0,\infty) \rightarrow [0,\infty)$ are continuous functions, satisfy $f(\lambda)g(\lambda) = \lambda,$ for all $\lambda \in [0,\infty)$, then $$ w_e\left(
\mathbb{A}\right)\leq w \left(\begin{bmatrix}
a_{ij}
\end{bmatrix}_{n\times n}\right),$$ where $ a_{ij}=\begin{cases}w_e(\mathbf{A_{ij}})\,\,\,\textit{ when i=j}\\
 \sqrt{w_e\left(f^2(|\mathbf{A_{ji}|})+g^2(|\mathbf{A_{ij}^*}|)\right)w_e\left(f^2(|\mathbf{A_{ij}|})+g^2(|\mathbf{A_{ji}^*}|)\right)}\,\,\,\textit{ when $i<j$}\\
 0\,\,\,\textit{ when $i>j$}.\end{cases}$
\label{them1}\end{theorem}
\begin{proof}
Let	$\mathbf{A_{ij}}=(A_{ij}^1,A_{ij}^2,\ldots,A_{ij}^d)\in\mathbb{B}^d(\mathscr{H}),$ $1\leq i,j\leq n,$ and 
	 $u=(x_1,x_2,\ldots,x_n)\in\oplus_{i=1}^{n}\mathscr{H}$ with $\|u\|=1$, i.e., $\|x_1\|^2+\|x_2\|^2+\ldots+\|x_n\|^2=1.$ Now,
\begin{eqnarray}
&&\left\|\langle\mathbb{A}u,u\rangle\right\|=\left\|\sum_{i,j=1}^{n}\left\langle \mathbf{A_{ij}} x_j,x_i\right\rangle\right\|\nonumber\\
&\leq& \left\|\sum_{i=1}^{n}\left\langle \mathbf{A_{ii}} x_i,x_i\right\rangle\right\|+\left\|\sum_{\underset{i\ne j}{i,j=1}}^{n}\left\langle \mathbf{A_{ij}} x_j,x_i\right\rangle\right\|\nonumber\\
&\leq& \sum_{i=1}^{n}\left\|\left\langle \mathbf{A_{ii}} x_i,x_i\right\rangle\right\|+\sum_{\underset{i< j}{i,j=1}}^{n}\left\|\left\langle \mathbf{A_{ij}} x_j,x_i\right\rangle+\left\langle \mathbf{A_{ji}} x_i,x_j\right\rangle\right\|\nonumber\\
&\leq& \sum_{i=1}^{n}\left\|\left\langle \mathbf{A_{ii}} x_i,x_i\right\rangle\right\|+\sum_{\underset{i< j}{i,j=1}}^{n}\left(\sum_{k=1}^{d}\left|\left\langle A_{ij}^{k} x_j,x_i\right\rangle+\left\langle A_{ji}^{k} x_i,x_j\right\rangle\right|^2\right)^\frac12\nonumber\\
&\leq& \sum_{i=1}^{n}\left\|\left\langle \mathbf{A_{ii}} x_i,x_i\right\rangle\right\|+\sum_{\underset{i< j}{i,j=1}}^{n}\left(\sum_{k=1}^{d}\left(\left|\left\langle A_{ij}^{k} x_j,x_i\right\rangle\right|+\left|\left\langle A_{ji}^{k} x_i,x_j\right\rangle\right|˘\right)^2\right)^\frac12\nonumber\\
&\leq& \sum_{i=1}^{n}w_e(\mathbf{A_{ii}})\left\|x_i\right\|^2\nonumber\\
&+&\sum_{\underset{i< j}{i,j=1}}^{n}\left(\sum_{k=1}^{d}\left(\left\|f(| A_{ij}^{k}|) x_j\right\|\left\|g(|A_{ij}^{k^*}|)x_i\right\|+\left\|f(| A_{ji}^{k}|) x_i\right\|\left\|g(|A_{ji}^{k^*}|)x_j\right\|\right)^2\right)^\frac12
\label{eq31}\end{eqnarray}
where the last inequality follows by using Lemma \ref{lm5}. Now, by Cauchy-Schwarz inequality we get 
\begin{eqnarray*}
&&\left(\sum_{k=1}^{d}\left(\left\|f(| A_{ij}^{k}|) x_j\right\|\left\|g(|A_{ij}^{k^*}|)x_i\right\|+\left\|f(| A_{ji}^{k}|) x_i\right\|\left\|g(|A_{ji}^{k^*}|)x_j\right\|\right)^2\right)^\frac12\\
&\leq&\left(\sum_{k=1}^{d}\left\langle \left(f^2(| A_{ij}^{k}|)+g^2(|A_{ji}^{k^*}|)\right)x_j, x_j\right\rangle\left\langle \left(f^2(| A_{ji}^{k}|)+g^2(|A_{ij}^{k^*}|)\right)x_i, x_i\right\rangle\right)^\frac12\\
&\leq&\left(\sum_{k=1}^{d}\left|\left\langle \left(f^2(| A_{ij}^{k}|)+g^2(|A_{ji}^{k^*}|)\right)x_j, x_j\right\rangle\right|^2\right)^\frac14\left(\sum_{k=1}^{d} \left|\left\langle\left(f^2(| A_{ji}^{k}|)+g^2(|A_{ij}^{k^*}|)\right)x_i, x_i\right\rangle\right|^2\right)^\frac14\\
&&\,\,\,\,\,\,\,(\textit{using Cauchy-Schwarz inequality})\\
&\leq& w_e^\frac12\left(f^2(|\mathbf{A_{ij}|})+g^2(|\mathbf{A_{ji}^*}|)\right)w_e^\frac12\left(f^2(|\mathbf{A_{ji}|})+g^2(|\mathbf{A_{ij}^*}|)\right)\|x_j\|\|x_i\|.
\end{eqnarray*}
Hence, from (\ref{eq31}), we obtain that
\begin{eqnarray*}
\left\|\langle\mathbb{A}u,u\rangle\right\|&\leq&\sum_{i=1}^{n} w_e(\mathbf{A_{ii}})\|x_i\|^2\\
&+&\sum_{\underset{i< j}{i,j=1}}^{n} w_e^\frac12\left(f^2(|\mathbf{A_{ij}|})+g^2(|\mathbf{A_{ji}^*}|)\right)w_e^\frac12\left(f^2(|\mathbf{A_{ji}|})+g^2(|\mathbf{A_{ij}^*}|)\right)\|x_{j}\|\|x_{i}\|\\
&=&\left\langle A|u|,|u|\right\rangle,
\end{eqnarray*}
 where $|u|=\begin{bmatrix}
 \|x_1\|  \\
 \|x_2\|  \\
 \vdots \\
 \|x_n\|
 \end{bmatrix}\in\mathbb{C}^n$ is an unit vector and $A=\begin{bmatrix}
 a_{ij}
 \end{bmatrix}_{n\times n}$ is an $n\times n$ complex matrix with 
$ a_{ij}=\begin{cases}w_e(\mathbf{A_{ij}})\,\,\,\textit{ when i=j}\\
 w_e^\frac12\left(f^2(|\mathbf{A_{ji}|})+g^2(|\mathbf{A_{ij}^*}|)\right)w_e^\frac12\left(f^2(|\mathbf{A_{ij}|})+g^2(|\mathbf{A_{ji}^*}|)\right)\,\,\,\textit{ when $i<j$}\\
 0\,\,\,\textit{ when $i>j$}.\end{cases}$\\
Therefore, \begin{eqnarray*}
	\left\|\langle\mathbb{A}u,u\rangle\right\|&\leq&\left\langle A|u|,|u|\right\rangle \,\leq \, w(A),
\end{eqnarray*}
holds for all $u\in \oplus_{i=1}^n\mathscr H$ with $\|u\|=1.$
 This completes the proof. 
\end{proof}

By Considering $ f(\lambda)=\lambda^\alpha $ and $ g(\lambda)=\lambda^{(1-\alpha)}$, $ 0\leq \alpha \leq 1 $ in Theorem \ref{them1}, we obtain the following corollary.

\begin{cor}
Let $\mathbb{A}=\begin{bmatrix}
\mathbf{A_{ij}}
\end{bmatrix}_{n\times n}$ be an $n\times n$ operator matrix, where  $\mathbf{A_{ij}}\in\mathbb{B}^d(\mathscr{H}),$ $1\leq i,j\leq n.$ Then $$ w_e\left(
\mathbb{A}\right)\leq w \left(\begin{bmatrix}
b_{ij}
\end{bmatrix}_{n\times n}\right),$$ where $ b_{ij}=\begin{cases}w_e(\mathbf{A_{ij}})\,\,\,\textit{ when i=j}\\
 w_e^\frac12\left(|\mathbf{A_{ji}|^{2\alpha}}+|\mathbf{A_{ij}^*}|^{2(1-\alpha)}\right)w_e^\frac12\left(|\mathbf{A_{ij}|^{2\alpha}}+|\mathbf{A_{ji}^*}|^{2(1-\alpha)}\right)\,\,\,\textit{ when $i<j$}\\
 0\,\,\,\textit{ when $i>j$},\end{cases}$ \\
 for all $\alpha,$  $0\leq \alpha \leq 1.$ In particular,
 $$ w_e\left(
\mathbb{A}\right)\leq w \left(\begin{bmatrix}
b'_{ij}
\end{bmatrix}_{n\times n}\right),$$  where $ b'_{ij}=\begin{cases}w_e(\mathbf{A_{ij}})\,\,\,\textit{ when i=j}\\
 \sqrt{w_e\left(|\mathbf{A_{ji}|}+|\mathbf{A_{ij}^*}|\right)w_e\left(|\mathbf{A_{ij}|}+|\mathbf{A_{ji}^*}|\right)}\,\,\,\textit{ when $i<j$}\\
 0\,\,\,\textit{ when $i>j$}.\end{cases}$
\label{cor1}\end{cor}

Since $w_e(\mathbf{A})\leq \|\mathbf{A}\|$  for every $\mathbf{A} \in \mathbb{B}^d(\mathscr{H})$ (see \eqref{T1}) and $w\left([a_{ij}]_{n\times n}\right)\leq w\left([b_{ij}]_{n\times n}\right),$ when $0\leq a_{ij}\leq b_{ij}$ for all $i,j$, the following corollaries are immediate from Theorem \ref{them1} and Corollary \ref{cor1}, respectively.

\begin{cor}
Let $\mathbb{A}=\begin{bmatrix}
\mathbf{A_{ij}}
\end{bmatrix}_{n\times n}$  be an $n\times n$ operator matrix, where $\mathbf{A_{ij}}\in\mathbb{B}^d(\mathscr{H}),$ $1\leq i,j\leq n.$ If $f, g : [0,\infty) \rightarrow [0,\infty)$ are continuous functions, satisfy $f(\lambda)g(\lambda) = \lambda,$ for all $\lambda \in [0,\infty)$, then $$ w_e\left(
\mathbb{A}\right)\leq w \left(\begin{bmatrix}
c_{ij}
\end{bmatrix}_{n\times n}\right),$$ where $ c_{ij}=\begin{cases}w_e(\mathbf{A_{ij}})\,\,\,\textit{ when i=j}\\
 \left\|f^2(|\mathbf{A_{ji}|})+g^2(|\mathbf{A_{ij}^*}|)\right\|^\frac12\left\|f^2(|\mathbf{A_{ij}|})+g^2(|\mathbf{A_{ji}^*}|)\right\|^\frac12\,\,\,\textit{ when $i<j$}\\
 0\,\,\,\textit{ when $i>j$}.\end{cases}$
\label{cor2}\end{cor}


\begin{cor}
Let $\mathbb{A}=\begin{bmatrix}
\mathbf{A_{ij}}
\end{bmatrix}_{n\times n}$  be an $n\times n$ operator matrix, where $\mathbf{A_{ij}}\in\mathbb{B}^d(\mathscr{H}),$ $1\leq i,j\leq n.$ Then $$ w_e\left(
\mathbb{A}\right)\leq w \left(\begin{bmatrix}
d_{ij}
\end{bmatrix}_{n\times n}\right),$$ where $ d_{ij}=\begin{cases}w_e(\mathbf{A_{ij}})\,\,\,\textit{ when i=j}\\
 \left\||\mathbf{A_{ji}|^{2\alpha}}+|\mathbf{A_{ij}^*}|^{2(1-\alpha)}\right\|^\frac12\left\||\mathbf{A_{ij}|^{2\alpha}}+|\mathbf{A_{ji}^*}|^{2(1-\alpha)}\right\|^\frac12\,\,\,\textit{ when $i<j$}\\
 0\,\,\,\textit{ when $i>j$},\end{cases}$\\
  for all $\alpha,\, 0\leq \alpha\leq 1.$
  In particular, 
  $$ w_e\left(
\mathbb{A}\right)\leq w \left(\begin{bmatrix}
d'_{ij}
\end{bmatrix}_{n\times n}\right),$$ where $ d'_{ij}=\begin{cases}w_e(\mathbf{A_{ij}})\,\,\,\textit{ when i=j}\\
 \left\||\mathbf{A_{ji}|}+|\mathbf{A_{ij}^*}|\right\|^\frac12\left\||\mathbf{A_{ij}|}+|\mathbf{A_{ji}^*}|\right\|^\frac12\,\,\,\textit{ when $i<j$}\\
 0\,\,\,\textit{ when $i>j$}.\end{cases}$
\label{cor3}\end{cor}


Now, it is well known that if $B=[b_{ij}]_{n\times n}$ is an $n\times n$ complex matrix with $b_{ij}\geq0$ for all $i,j=1,2,...,n,$ then $$
w(B)=w\left(\begin{bmatrix}
	\frac{b_{ij}+b_{ji}}{2}
\end{bmatrix}_{n\times n}\right),$$ see in \cite[p. 44]{H}. By employing this argument, the bounds in Corollary \ref{cor1} and \ref{cor3} can be written as in the following remarks, respectively.



\begin{remark}
Let $\mathbb{A}=\begin{bmatrix}
\mathbf{A_{ij}}
\end{bmatrix}_{n\times n}$  be an $n\times n$ operator matrix, where $\mathbf{A_{ij}}\in\mathbb{B}^d(\mathscr{H}),$ $1\leq i,j\leq n.$ Then $$ w_e\left(
\mathbb{A}\right)\leq w \left(\begin{bmatrix}
e_{ij}
\end{bmatrix}_{n\times n}\right),$$ where $ e_{ij}=\begin{cases}w_e(\mathbf{A_{ij}})\,\,\,\textit{ when i=j}\\
 \frac12w_e^\frac12\left(|\mathbf{A_{ji}|^{2\alpha}}+|\mathbf{A_{ij}^*}|^{2(1-\alpha)}\right)w_e^\frac12\left(|\mathbf{A_{ij}|^{2\alpha}}+|\mathbf{A_{ji}^*}|^{2(1-\alpha)}\right)\,\,\,\textit{ when $i\ne j$}.\end{cases}$ \\\\
for all $\alpha,\, 0\leq \alpha \leq 1.$ In particular,  
 $$ w_e\left(
\mathbb{A}\right)\leq w \left(\begin{bmatrix}
e'_{ij}
\end{bmatrix}_{n\times n}\right),$$  where $ e'_{ij}=\begin{cases}w_e(\mathbf{A_{ij}})\,\,\,\textit{ when i=j}\\
\frac12 w_e^\frac12\left(|\mathbf{A_{ji}|}+|\mathbf{A_{ij}^*}|\right)w_e^\frac12\left(|\mathbf{A_{ij}|}+|\mathbf{A_{ji}^*}|\right)\,\,\,\textit{ when $i\ne j$}.\end{cases}$
\label{cor4}\end{remark}

\begin{remark}
Let $\mathbb{A}=\begin{bmatrix}
\mathbf{A_{ij}}
\end{bmatrix}_{n\times n}$ be an $n\times n$ operator matrix, where $\mathbf{A_{ij}}\in\mathbb{B}^d(\mathscr{H}),$ $1\leq i,j\leq n.$ Then $$ w_e\left(
\mathbb{A}\right)\leq w \left(\begin{bmatrix}
f_{ij}
\end{bmatrix}_{n\times n}\right),$$ where $ f_{ij}=\begin{cases}w_e(\mathbf{A_{ij}})\,\,\,\textit{ when i=j}\\
 \frac12\left\||\mathbf{A_{ji}|^{2\alpha}}+|\mathbf{A_{ij}^*}|^{2(1-\alpha)}\right\|^\frac12\left\||\mathbf{A_{ij}|^{2\alpha}}+|\mathbf{A_{ji}^*}|^{2(1-\alpha)}\right\|^\frac12\,\,\,\textit{ when $i\ne j$}
. \end{cases}$\\\\
  for all $\alpha,\, 0\leq \alpha \leq 1.$ 
  In particular,  
  $$ w_e\left(
\mathbb{A}\right)\leq w \left(\begin{bmatrix}
f'_{ij}
\end{bmatrix}_{n\times n}\right),$$ where $ f'_{ij}=\begin{cases}w_e(\mathbf{A_{ij}})\,\,\,\textit{ when i=j}\\
\frac12 \left\||\mathbf{A_{ji}|}+|\mathbf{A_{ij}^*}|\right\|^\frac12\left\||\mathbf{A_{ij}|}+|\mathbf{A_{ji}^*}|\right\|^\frac12\,\,\,\textit{ when $i\ne j$}
 .\end{cases}$
\label{cor5}\end{remark}

Considering $n=2$ in Remark \ref{cor4} and Remark \ref{cor5}, we develop the following bounds for the Euclidean operator radius of $2\times 2 $ operator matrices whose entries are $d$-tuple operators.

\begin{cor}
	Let  $\mathbf{A}$, $\mathbf{B}$,  $\mathbf{C}$,  $\mathbf{D}\in\mathbb{B}^d(\mathscr{H})$. Then $$ w_e\left(\begin{bmatrix}
		\mathbf{A} & \mathbf{B} \\
		\mathbf{C} & \mathbf{D}
	\end{bmatrix}\right)\leq \frac12  \left( w_e(\mathbf{A})+w_e(\mathbf{D})+\sqrt{\left(w_e(\mathbf{A})-w_e(\mathbf{D})\right)^2+\beta^2}\right),$$ where  $\beta= \sqrt{w_e^{}\left((|\mathbf{B}|+|\mathbf{C^*}|)\right)w_e^{}\left((|\mathbf{C}|+|\mathbf{B^*}|)\right)}$.  
\label{cor6}\end{cor}

\begin{cor}
Let $\mathbf{A}$, $\mathbf{B}$,  $\mathbf{C},$  $\mathbf{D}\in\mathbb{B}^d(\mathscr{H})$. Then $$ w_e\left(\begin{bmatrix}
		\mathbf{A} & \mathbf{B} \\
		\mathbf{C} & \mathbf{D}
	\end{bmatrix}\right)\leq \frac12  \left( w_e(\mathbf{A})+w_e(\mathbf{D})+\sqrt{\left(w_e(\mathbf{A})-w_e(\mathbf{D})\right)^2+\gamma^2}\right),$$ where $\gamma=\left\||\mathbf{B}|+|\mathbf{C^*}|\right\|^\frac12 \left\||\mathbf{C}|+|\mathbf{B^*}|\right\|^\frac12$.  
\label{cor7}\end{cor}

In \cite{SPK1}, we proved that $w_e\left(\mathbf{A}^n\right)\leq \sqrt{d}\ {w_e^n\left(\mathbf{A}\right)}$ for every $\mathbf{A}\in\mathbb{B}^d(\mathscr{H})$ and for every positive integer $n.$
Using this power inequality, we develop an upper bound for the joint numerical radius of the product of two $d$-tuple operators.

\begin{theorem}
If $\mathbf{B}$, $\mathbf{C}\in\mathbb{B}^d(\mathscr{H})$, then $$w_e(\mathbf{BC})\leq \frac{\sqrt{d}}{4} w_e\left((|\mathbf{B}|+|\mathbf{C^*}|)\right)w_e\left((|\mathbf{C}|+|\mathbf{B^*}|)\right)
. $$
\label{th15}\end{theorem}
\begin{proof}
Following \cite[Lemma 3.1]{SPK1}, 
 we have \begin{eqnarray*}
  w_e(\mathbf{BC}) \leq  \max\{w_e(\mathbf{BC}),w_e(\mathbf{CB})\}
  &=&w_e \left(\begin{bmatrix}
 \mathbf{BC} & 0 \\
0 &\mathbf{CB} 
\end{bmatrix}\right)\\ &=&w_e\left(\begin{bmatrix}
0 & \mathbf{B} \\
\mathbf{C} & 0
\end{bmatrix}^2\right)\nonumber\\
&\leq&\sqrt{d}\ w_e^2\left(\begin{bmatrix}
0 & \mathbf{B} \\
\mathbf{C} & 0
\end{bmatrix}\right).\end{eqnarray*}
 Using Corollary \ref{cor6}, we have 
 $$ w_e(\mathbf{BC})\leq \frac{\sqrt{d}}{4} w_e((|\mathbf{B}|+|\mathbf{C^*}|))w_e((|\mathbf{C}|+|\mathbf{B^*}|)). $$ 
\end{proof}  

Finally, using the above theorem we develop an upper bound for the Euclidean operator radius of $d$-tuple operators.

\begin{cor}
If $\mathbf{A} \in {\mathbb{B}^d(\mathscr{H})}$, then 
$$ w_e(\mathbf{A})\leq \frac{\sqrt{d}}{4} w_e\left(|\mathbf{A}|^{1-t}+|\mathbf{A}|^t\right)w_e\left(|\mathbf{A}|^t+|\mathbf{A^*}|^{1-t}\right),$$  for all $t,\ 0\leq t\leq 1.$
\label{theo1}\end{cor}
 \begin{proof}
 Suppose $\mathbf{A}=(A_1,A_2, \ldots,A_d)$, $\mathbf{B}=(B_1,B_2, \ldots,B_d)$,  $\mathbf{C}=(C_1,C_2, \ldots,C_d)\in \mathbb{B}^d(\mathscr{H}).$
 Let $A_k=U_k|A_k|$ be the polar decomposition of $A_k$ for all $k=1,2,\ldots,d.$ Considering $B_k=U_k|A_k|^{1-t}$ and $C_k=|A_k|^t$ in Theorem \ref{th15}, we get the desired result. 
 \end{proof}
 
\noindent  \textbf{Statements of declaration.}\\
The authors have no relevant financial or non-financial interests to disclose.

\bibliographystyle{amsplain}

\begin{thebibliography}{99}
	



\bibitem{BDMP} P. Bhunia, S.S. Dragomir, M.S. Moslehian, K. Paul; Lectures on Numerical Radius Inequalities, Infosys Science Foundation Series in Mathematical Sciences, Springer, Cham, 2022. xii+209 pp. ISBN: 978-3-031-13669-6; 978-3-031-13670-2

\bibitem{Bhunia1} P. Bhunia; Sharper bounds for the numerical radius of $\lowercase{n}\times \lowercase{n}$ operator matrices. https://arxiv.org/abs/2303.10392v1

\bibitem{Bhunia2} P. Bhunia; Improved bounds for the numerical radius via polar decomposition of operators. https://arxiv.org/abs/2303.03051






\bibitem{Bhunia_LAMA_2022} P. Bhunia, K. Paul; Some improvements of numerical radius inequalities of operators and operator matrices, Linear Multilinear Algebra 70 (2022), no. 10, 1995--2013.

\bibitem{Bhunia_LAA_2021} P. Bhunia, K. Paul;  Development of inequalities and characterization of equality conditions for the numerical radius, Linear Algebra Appl. 630 (2021), 306--315. 

\bibitem{Bhunia_AM_2021}  P. Bhunia, K. Paul; Furtherance of numerical radius inequalities of Hilbert space operators, Arch. Math. (Basel) 117 (2021), no. 5, 537--546.

\bibitem{Bhunia_RIM_2021} P. Bhunia, K. Paul; Proper improvement of well-known numerical radius inequalities and their applications, Results Math. 76 (2021), no. 4, Paper No. 177, 12 pp. 

\bibitem{Bhunia_BSM_2021} P. Bhunia, K. Paul; New upper bounds for the numerical radius of Hilbert space operators, Bull. Sci. Math. 167 (2021), Paper No. 102959, 11 pp. 

\bibitem{Bhunia3} P. Bhunia; Numerical radius inequalities of bounded linear operators and $(\alpha,\beta)$-normal operators. https://arxiv.org/abs/2301.03877


\bibitem{Bhunia_LAA_2019}  P. Bhunia, S. Bag, K. Paul; Numerical radius inequalities and its applications in estimation of zeros of polynomials, Linear Algebra Appl. 573 (2019), 166--177.





\bibitem{B} M.L. Buzano; Generalizzatione della disuguaglianza di Cauchy-Schwarz, Rend. Semin. Mat. Univ. Politech. Torino 31(1971/73) (1974) 405--409. 





\bibitem{SD} S.S. Dragomir; Some inequalities for the Euclidean operator radius of two operators in Hilbert spaces, Linear Algebra Appl. 419 (2006), 256--264.

\bibitem{H} R.A. Horn, C.R. Johnson; Topics in Matrix Analysis, Cambridge University Press,
Cambridge, 1991.



\bibitem{SPK1} S. Jana, P. Bhunia, K. Paul; Euclidean operator radius inequalities of $d$-tuple operators and operator matrices, (2023) arXiv:2304.08033v1.

\bibitem{SPK} S. Jana, P. Bhunia, K. Paul; Euclidean operator radius inequalities of a pair of bounded linear operators and their applications. Bull. Braz. Math. Soc. (N.S.) 54 (2023), no. 1, Paper No. 1, 14 pp.





\bibitem{Kittaneh_LAMA_2023} F. Kittaneh, H.R. Moradi, M. Sababheh; Sharper bounds for the numerical radius, Linear Multilinear Algebra, (2023). 
https://doi.org/10.1080/03081087.2023.2177248

\bibitem{Kittaneh_STD_2005}  F. Kittaneh; Numerical radius inequalities for Hilbert space operators, Studia Math. 168 (2005), no. 1, 73--80.

\bibitem{Kittaneh_2003} F. Kittaneh; Numerical radius inequality and an estimate for the numerical radius of the Frobenius companion matrix, Studia Math. 158 (2003), no. 1, 11--17.

\bibitem{FK} F. Kittaneh; Notes on some inequalities for Hilbert space operators, {Publ. Res. Inst. Math. Sci.} 24 (1988), no. 2, 283--293.





\bibitem{MSS} M.S. Moslehian, M. Sattari, K. Shebrawi; Extensions of Euclidean operator radius inequalities, Math. Scand. 120 (2017), no. 1, 129-144.

\bibitem{M} C.A. McCarthy; Cp, Israel J. Math. 5 (1967), 249–271.



\bibitem{P} G. Popescu; Unitary invariants in multivariable operator theory, Mem. Amer. Math. Soc. 200 (2009), no. 941, vi+91 pp. ISBN: 978-0-8218-4396-3

\bibitem{Sahoo} S. Sahoo, N.C. Rout, M. Sababheh; Some extended numerical radius inequalities, Linear Multilinear Algebra 69 (2021), no. 5, 907--920.


\bibitem{VK} M.P. Vasi\'c,  D.J. Ke\^cki\'c; Some inequalities for complex numbers, Math. Balkanica 1 (1971) 282--286.


\bibitem{GW} P.Y. Wu, H.-L. Gau; Numerical ranges of Hilbert space operators, Encyclopedia of Mathematics and its Applications, 179. Cambridge University Press, Cambridge, 2021. xviii+483 pp. ISBN: 978-1-108-47906-6 47-02 



\end{thebibliography}

\end{document}